\documentclass[12pt]{amsart}
\usepackage{latexsym,amssymb,amsmath,mathtools,hyperref}
\usepackage{xcolor}

\numberwithin{equation}{section}
\newtheorem {theorem}[equation]{Theorem}
\newtheorem*{theorem*}{Theorem}
\newtheorem {lemma}[equation]{Lemma}
\newtheorem {corollary} [equation]      {Corollary}
\newtheorem*{corollary*}{Corollary}
\newtheorem {proposition}[equation]     {Proposition}

\theoremstyle{definition}
\newtheorem {definition}[equation]{Definition}

\theoremstyle{remark}
\newtheorem {remark}[equation]{Remark}
\newtheorem {example}[equation]{Example}
\newtheorem {notation}[equation]{Notation}

\newcommand{\flag}{\mathrm{Flags}}
\newcommand{\init}{\mathrm{in}}

\newcommand{\exampleqed}{%
    {   
    \renewcommand{\qedsymbol}{$\lozenge$}%
    \qed%
    }
}

\author[M. Cummings]{Mike Cummings}
\author[S. Da Silva]{Sergio Da Silva}
\author[M. Harada]{Megumi Harada}
\author[J. Rajchgot]{Jenna Rajchgot}

\address[Mike Cummings]{
    Dept.\ of Mathematics and Statistics, 
    McMaster University, 
    1280 Main Street West, 
    Hamilton, Ontario L8S 4K1, Canada}
\email{cummim5@mcmaster.ca}

\address[Sergio Da Silva]{
    Dept.\ of Mathematics and Economics, 
    Virginia State University, 
    1 Hayden Drive,
    Petersburg, Virginia 23806, USA}
\email{sdasilva@vsu.edu, smd322@cornell.edu}

\address[Megumi Harada]{
    Dept.\ of Mathematics and Statistics, 
    McMaster University, 
    1280 Main Street West, 
    Hamilton, Ontario L8S 4K1, Canada}
\email{haradam@mcmaster.ca}

\address[Jenna Rajchgot]{
    Dept.\ of Mathematics and Statistics, 
    McMaster University, 
    1280 Main Street West, 
    Hamilton, Ontario L8S 4K1, Canada}
\email{rajchgot@math.mcmaster.ca}

\title[Gr\"obner geometry for regular nilpotent Hessenberg Schubert cells]{Gr\"obner geometry for \\regular nilpotent Hessenberg Schubert cells}

\date{\today}

\begin{document}

\begin{abstract}
A regular nilpotent Hessenberg Schubert cell is the intersection of a regular nilpotent Hessenberg variety with a Schubert cell. In this paper, we describe a set of minimal generators of the defining ideal of a regular nilpotent Hessenberg Schubert cell in the type $A$ setting. We show that these minimal generators are a Gr\"obner basis for an appropriate lexicographic monomial order.
As a consequence, we obtain a new computational-algebraic proof, in type $A$, of Tymoczko's result that regular nilpotent Hessenberg varieties are paved by affine spaces.
In addition, we prove that these defining ideals are complete intersections, are geometrically vertex decomposable, and compute their Hilbert series. We also produce a Frobenius splitting of each Schubert cell that compatibly splits all of the regular nilpotent Hessenberg Schubert cells contained in it.
This work builds on, and extends, work of the second and third author on defining ideals of intersections of regular nilpotent Hessenberg varieties with the (open) Schubert cell associated to the Bruhat-longest permutation.
\end{abstract}

\maketitle

\section{Introduction} \label{sec:intro}

Hessenberg varieties are subvarieties of the full flag variety $\flag(\mathbb{C}^n)$.\footnote{In this manuscript, we restrict to the case of Lie type $A$, i.e., when the flag variety corresponds to the group $GL_n(\mathbb{C})$ (or $SL_n(\mathbb{C})$). Much can also be said about other Lie types.} They were first introduced to algebraic geometers by De Mari, Procesi, and Shayman \cite{DeMPS}, and their study lies in the intersection of algebraic geometry, representation theory, and combinatorics, among other research areas.  
For example: these varieties arise in the study of quantum cohomology of flag varieties, they are generalizations of the Springer fibers which arise in geometric representation theory, total spaces of families of suitable Hessenberg varieties support interesting integrable systems \cite{AbeCrooks}, and the study of some of their cohomology rings suggests that there is a rich Hessenberg analogue to the theory of Schubert calculus on $\flag(\mathbb{C}^n)$. 

The motivation of the present paper stems mainly from Schubert geometry.
Moreover, this manuscript can and should be viewed as a natural companion paper to the recent work \cite{DSH} of the second and third authors. Let us briefly recall the setting.
The main objects of discussion are the \textbf{local patches of Hessenberg varieties} --- i.e., intersections of Hessenberg varieties with certain choices of affine Zariski-open subsets of $\flag(\mathbb{C}^n)$. (Throughout this manuscript, we restrict to the so-called regular nilpotent case -- to be defined precisely below in Section~\ref{subsec:Hess}.) 
In particular, we may study the \textbf{local Hessenberg patch ideal}, denoted $I_{w_0,h}$, for the special case of a \textbf{regular nilpotent Hessenberg variety} (to be defined in Section~\ref{subsec:Hess}), intersected with the affine coordinate chart $w_0B^-B$ of $\flag(\mathbb{C}^n) \cong GL_n(\mathbb{C})/B$ centered at the permutation flag corresponding to the maximal element 
$w_0$ in $S_n$; this is the case studied in \cite{DSH}.  (Here $B$ is the usual Borel subgroup of upper-triangular invertible matrices in $GL_n(\mathbb{C})$ and $B^-$ is the opposite Borel of lower-triangular invertible matrices.) 
This affine coordinate chart is also a \textbf{Schubert cell}, $B(\text{Id})B/B$ of the identity permutation $\text{Id}\in S_n$.

The main contribution of this manuscript is that we prove analogues of the results of \cite{DSH} for \emph{all} possible \textbf{regular nilpotent Hessenberg Schubert cells} -- that is, for all possible intersections of regular nilpotent Hessenberg varieties, denoted $\text{Hess}(\mathsf N,h)$, with Schubert cells $BwB/B$, for any $w\in S_n$. 
It turns out that the perspectives used in \cite{DSH} can be treated in a similar manner to the $w_0$-chart, as long as we restrict our attention to the Hessenberg Schubert cells (instead of looking at the entire local Hessenberg patch). A rough statement of our main results, which is stated more precisely as Corollary~\ref{cor: init Jwh} and Theorem \ref{thm:subsetGB}, is as follows:

\begin{theorem*}
Let $n$ be a positive integer and  
let $h: [n] \to [n]$ be an indecomposable Hessenberg function.
Let $w\in S_n$ and suppose that the Hessenberg Schubert cell $\mathrm{Hess}(\mathsf N,h)\cap BwB/B$ is non-empty. 
Let $J_{w,h}\subseteq \mathbb{C}[BwB/B]$ denote the (radical) defining ideal of $\mathrm{Hess}(\mathsf N,h)\cap BwB/B$. Then there exists a lexicographic monomial order $<$ such that $\mathrm{in}_< (J_{w,h})$ is generated by $|\lambda_h|$-many indeterminates, where $\lambda_h$ is a partition associated to the Hessenberg function $h$. Moreover, there is a natural list of generators of $J_{w,h}$ which form a Gr\"obner basis for $J_{w,h}$ with respect to $<$. 
\end{theorem*}

The above results easily imply that each non-empty Hessenberg Schubert cell is isomorphic to an affine space $\mathbb{A}^{r_w - |\lambda_h|}$ where $r_w$ is the dimension of the Schubert cell $BwB/B$, $w\in S_n$. Using this, in Corollary \ref{cor:PavingByAffines}, we recover the known result that each regular nilpotent Hessenberg variety is paved by affine spaces. This was initially proved by Tymoczko in \cite{Tymoczko} and generalized to the regular (not necessarily nilpotent) case by Precup in \cite{Precup}. 
Tymoczko and Precup's proofs use the language of Lie theory (and apply beyond type $A$) and are delicate combinatorial arguments. Thus, one contribution of this paper is a new, quick proof of this paving-by-affines result in the type $A$ regular nilpotent case, using purely computational-algebraic methods.  

As in \cite{DSH}, our main theorem has a number of immediate consequences, including geometric vertex decomposability, Frobenius splitting, and Hilbert series. These applications are discussed in detail, and more precisely, in Section~\ref{sec:applications}. 

\begin{corollary*}
Let $n$ be a positive integer and  
let $h: [n] \to [n]$ be an indecomposable Hessenberg function.
Let $w\in S_n$ and suppose that the Hessenberg Schubert cell $\mathrm{Hess}(\mathsf N,h)\cap BwB/B$ is non-empty. 
Let $J_{w,h}\subseteq \mathbb{C}[BwB/B]$ denote the (radical) defining ideal of $\mathrm{Hess}(\mathsf N,h)\cap BwB/B$. Then, 
\begin{enumerate}
    \item each $J_{w,h}\subseteq \mathbb{C}[BwB/B]$ is a complete intersection ideal;
    \item the Hilbert series of $\mathbb{C}[BwB/B]/J_{w,h}$,  with respect to $\mathbb{Z}$-grading induced by the natural circle action on the Hessenberg Schubert cell, is given by 
    \[ 
    H_{R/J_{w,h}}(t)
    = \frac{
    \displaystyle
    \prod_{\substack{k > h(\ell) \\ v_w(k) > v_w(\ell) + 1}}
    (1 - t^{v_w(k) - v_w(\ell) - 1})
    }{
    \displaystyle
    \prod_{\substack{i < w(j) \\ j < w^{-1}(i)}} 
    \left( 1 - t^{w(j) - i} \right)
    }.
    \]
    
    \item each ideal $J_{w,h}$ is geometrically vertex decomposable;
    \item upon replacing $\mathbb{C}$ by a algebraically closed field $\mathbb{K}$ of positive characteristic, there is a natural Frobenius splitting of $\mathbb{K}[BwB/B]$ which compatibly splits all ideals $J_{w,h}\subseteq \mathbb{K}[BwB/B]$.
\end{enumerate}
\end{corollary*}

We briefly describe the contents of the manuscript. In Section~\ref{sec:background} we recall the key definitions and background. In Section~\ref{sec:GB-HessSchubertCells} we introduce coordinates on the intersections of Schubert cells with regular nilpotent Hessenberg varieties, and define their local defining ideals via explicit relations. Then in Section~\ref{sec: Grobner for Hess Schubert cells} we prove our main results, and in Section~\ref{sec:applications} we give a detailed discussion of the many immediate applications of our main results, as itemized above.

\subsection*{Remark on the field.} 
For simplicity, as well as consistency with some of the literature that we cite, we work over $\mathbb C$. We note, however, that all of our computational-algebraic arguments work more generally. For example, our Gr\"obner basis arguments work over arbitrary base fields and our Frobenius splitting results hold over perfect fields of positive characteristic.

\subsection*{Acknowledgements} 
We thank the anonymous referee for a careful reading of this manuscript and their helpful comments.
Cummings was supported in part by an NSERC CGS-M and The Milos Novotny Fellowship.
Da Silva was supported in part by an NSERC Postdoctoral Fellowship of Canada.
Harada was supported in part by NSERC Discovery Grant 2019-06567 and a Canada Research Chair Tier 2 award.
Rajchgot was supported in part by NSERC Discovery Grant 2017-05732.

\section{Background} \label{sec:background}

In this section we recall some background on flag varieties, regular nilpotent Hessenberg varieties, Hessenberg patch ideals, and a certain circle action on regular nilpotent Hessenberg varieties.
We follow \cite{DSH} as our main reference for notational conventions.
For the purposes of this manuscript we restrict to Lie groups of type $A$, although Hessenberg varieties can be defined in arbitrary Lie type.

\subsection{Flag varieties} \label{subsec:Flags}

The {\bf (full) flag variety}
${\rm Flags}(\mathbb C^n)$ is the set of sequences of nested subspaces of $\mathbb C^n$ \[ {\rm Flags}(\mathbb C^n) \coloneqq \{ V_\bullet = (0 = V_0 \subsetneq V_1 \subsetneq V_2 \subsetneq \cdots \subsetneq V_n = \mathbb C^n) \mid \dim_{\mathbb C}(V_i) = i \}. \]
Let $B$ denote the Borel subgroup of $GL_n(\mathbb C)$ of upper-triangular invertible matrices and let $U^-$ denote the subgroup of $GL_n(\mathbb C)$ of lower-triangular matrices with $1$'s along the diagonal.
By representing $V_\bullet \in {\rm Flags}(\mathbb C^n)$ as an element of $GL_n(\mathbb C)$ by the matrix whose first $i$ columns span $V_i$, we can identify ${\rm Flags}(\mathbb C^n)$ with $GL_n(\mathbb C)/B$.
Then $U^-B$, the set of cosets $uB$ for $u \in U^-$, can be viewed as a coordinate chart on $GL_n(\mathbb C)/B \cong {\rm Flags}(\mathbb C^n)$, which we show precisely below.
In fact, $U^-B$ is an open dense subset of ${\rm Flags}(\mathbb C^n)$.

Let $w \in S_n$ be a permutation of the symmetric group on $n$ letters.
By slight abuse of notation, we will often take $w$ to also mean its corresponding permutation matrix.
Here we take the convention that we record the permutation ``along columns'' to obtain the corresponding permutation matrix; that is, we write 
\begin{equation} \label{eqn:permutationConvention}
    w = \begin{bmatrix} | & | & & | \\ e_{w(1)} & e_{w(2)} & \cdots & e_{w(n)} \\ | & | & & | \end{bmatrix}, 
\end{equation}
where $e_j$ is the $j^{\rm th}$ standard basis vector, written as a column vector, with a 1 in the $j^{\rm th}$ entry and 0's elsewhere.

Given the above notation, we define an {\bf open cell} (a coordinate chart) in $GL_n(\mathbb C)/B$ containing the permutation $w \in S_n$ by the translation
\begin{equation*}
    \mathcal N_w \coloneqq wU^-B \subseteq GL_n(\mathbb C)/B.
\end{equation*}
To describe $\mathcal{N}_w$ more concretely, let $M$ denote a generic element of $U^-$, 
\begin{equation} \label{eqn:M}
    M \coloneqq \begin{bmatrix}
    1 \\
    \star & 1 \\
    \vdots & \vdots & \ddots \\
    \star & \star & \cdots & 1 \\
    \star & \star & \cdots & \star & 1
    \end{bmatrix},
\end{equation}
where the $\star$'s are complex numbers. Since the entries below the diagonal are free of constraints, it is clear that $U^-$, the set of all such $M$'s, is isomorphic to affine space $\mathbb{A}^{n(n-1)/2}$. 
For any $w \in S_n$, the map $M \mapsto wMB \in GL_n(\mathbb C)/B$ is an embedding $U^- \cong \mathbb A^{n(n-1)/2} \xrightarrow{\cong} \mathcal N_w$, showing that $\mathcal{N}_w$ is isomorphic to affine space as well. Henceforth, we can identify a point in $\mathcal N_w$ with a translation by $w$ of a matrix of the form~\eqref{eqn:M}. 
In particular, a point in $\mathcal N_w$ is uniquely identified by a matrix $wM$ whose $(i,j)$-th entries are given by 
\begin{equation} \label{eqn:wM}
    [wM]_{i,j} = \begin{cases}
    1 & \text{if } i = w(j), \\
    0 & \text{if } j > w^{-1}(i), \\
    x_{i,j} & \text{otherwise, for some } x_{i,j} \in \mathbb{C}.
    \end{cases}
\end{equation}

\begin{example} \label{eg:w_0M}
    Denote by $w_0 \in S_n$ the longest permutation in Bruhat order.
    If $n = 4$, then $w_0 = 4321$ in one-line notation and 
    \begin{equation*}
        w_0M = \begin{bmatrix} x_{1,1} & x_{1,2} & x_{1,3} & 1 \\ x_{2,1} & x_{2,2} & 1 & 0 \\ x_{3,1} & 1 & 0 & 0 \\ 1 & 0 & 0 & 0 \end{bmatrix}.
    \end{equation*}
    \exampleqed
\end{example}

It follows from \eqref{eqn:wM} that there is an isomorphism between the coordinate ring of $\mathcal{N}_w$ and the polynomial ring in $n(n-1)/2$ indeterminates, the $x_{i,j}$ that satisfy $j<w^{-1}(i)$. We denote this polynomial ring by $\mathbb{C}[\mathbf{x}_w]$ where $\mathbf{x}_w$ denotes the finite set of indeterminates labelled $x_{i,j}$ for $j<w^{-1}(i)$.

\subsection{Hessenberg varieties and Hessenberg patch ideals} \label{subsec:Hess}

We next review the definition of a Hessenberg variety.
Let $[n] \coloneqq \{1,2, \ldots, n\}$.
A function $h:[n] \to [n]$ is a {\bf Hessenberg function} if $h(i) \ge i$ for all $i \in [n]$ and $h(i+1) \ge h(i)$ for all $i \in [n-1]$.
An {\bf indecomposable} Hessenberg function additionally satisfies $h(i) \ge i+1$ for all $i \in [n-1]$.
To a Hessenberg function $h$, we may associate the \textbf{Hessenberg subspace} of $\mathfrak{gl}_n(\mathbb{C})$, defined by 
\begin{equation*}
H(h) := \{ (a_{i,j})_{i,j \in [n]} \, \mid \, a_{i,j}=0 \, \textup{if} \, i>h(j)\},
\end{equation*}
where $\mathfrak{gl}_n(\mathbb{C})$ is the Lie algebra of $GL_n(\mathbb{C})$ consisting of $n\times n$ matrices with entries in $\mathbb{C}$. 
Given an indecomposable Hessenberg function $h$ and a linear operator $A:\mathbb C^n \to \mathbb C^n$, we define the {\bf Hessenberg variety associated to $A$ and $h$} to be the following subvariety of ${\rm Flags}(\mathbb C^n)$, 
\begin{equation*}
    {\rm Hess}(A,h) \coloneqq \{ (V_0 \subsetneq \cdots \subsetneq V_n) \in {\rm Flags}(\mathbb C^n) \mid AV_i \subseteq V_{h(i)} \text{ for all } i \in [n] \} . 
\end{equation*}
Equivalently, viewing a flag as a coset $MB$ in $GL_n(\mathbb{C})/B$ and representing it by a matrix $M$ in the coset $MB$, and viewing $A$ as an element in $\mathfrak{gl}_n(\mathbb{C})$,
we may also describe $\mathrm{Hess}(A,h)$ as 
\begin{equation} \label{eq: matrix eq for Hess}
\mathrm{Hess}(A,h) \coloneqq \{ MB \in GL_n(\mathbb{C})/B \, \mid \, M^{-1}AM \in H(h) \}.
\end{equation}
For more on this equivalence for all regular Hessenberg varieties, see \cite[Section 4]{ITW}.

In this manuscript we focus on the case that $A$ is a regular nilpotent operator.
Specifically, we restrict attention to the Hessenberg varieties ${\rm Hess}(A, h)$ where $A=\mathsf N$ is defined by
\begin{equation} \label{eqn:N}
    \mathsf N \coloneqq 
    \begin{bmatrix} 
    0 & 1 & 0 & \cdots & 0 \\
    0 & 0 & 1 & \cdots & 0 \\
    \vdots &&& \ddots & \vdots \\
    0 & 0 & 0 & \cdots & 1 \\
    0 & 0 & 0 & \cdots & 0
    \end{bmatrix}.
\end{equation}
The remainder of this section will be devoted to reviewing the defining equations for the affine variety $\mathcal N_w \cap {\rm Hess}(\mathsf N, h)$ from \cite[Section 2.2]{DSH}.
In the next section, we will describe equations which define the intersection of ${\rm Hess}(\mathsf N, h)$ with a \emph{Schubert cell}.

In the definition below, as well as in the discussion that follows, we view $M$ as a matrix with entries $0,1,$ or an indeterminate in the set $\mathbf{x}_w = \{x_{i,j}\}$ as in Section~\ref{subsec:Flags}, and view each matrix entry of $(wM)^{-1}\mathsf N(wM)$ as an element in $\mathbb{C}[\mathbf{x}_w]$. 
Observe that since $\det(wM) = \pm 1$, the entries of $(wM)^{-1}$ are polynomials in the entries of $wM$, and in particular, are not rational functions.

\begin{definition}[{\cite[Definition 3.3]{ADGH}}] \label{def:HessPatchIdealAndf^w}
    Let $h$ be a Hessenberg function. Let $w \in S_n$ and $k, \ell \in [n]$.
    Following the notation of \eqref{eqn:wM} and \eqref{eqn:N}, define polynomials in $\mathbb C[\mathbf x_w]$ by
    \begin{equation*}
        f^w_{k,\ell} \coloneqq [(wM)^{-1}\mathsf N (wM)]_{k,\ell}.
    \end{equation*}
    For an indecomposable Hessenberg function $h$, we define the {\bf Hessenberg patch ideal corresponding to $w$ and $h$} to be the ideal 
    \begin{equation*}
        I_{w,h} \coloneqq \langle f^w_{k,\ell} \mid k > h(\ell) \rangle \subseteq \mathbb C[\mathbf x_w].
    \end{equation*}
\end{definition}

From~\eqref{eq: matrix eq for Hess} it follows that the $f^w_{k,\ell}$ with $k > h(\ell)$ must vanish on $\mathrm{Hess}(\mathsf{N},h)$ and they (set-theoretically) cut out $\mathrm{Hess}(\mathsf{N},h)$ in $\mathcal{N}_w$. In fact, the following is known. 

\begin{proposition}[{\cite[Proposition 3.7]{ADGH}}]
The Hessenberg patch ideals $I_{w,h}$ defined above are radical and are the defining ideals of the affine variety $\mathcal N_w \cap {\rm Hess}(\mathsf N, h)$. 
\end{proposition} 

\begin{example} \label{eg:patchIdeal}
    Consider $w_0 \in S_4$ as in Example \ref{eg:w_0M}.
    Then,
    \begin{equation*}
        (w_0M)^{-1}\mathsf N (w_0M) = \begin{bmatrix}  
            0 & 0 & 0 & 0 \\
            1 & 0 & 0 & 0 \\
            -x_{2,2}+x_{3,1} & 1 & 0 & 0 \\
            -x_{1,2} + x_{1,3}(x_{2,2} - x_{3,1}) + x_{2,1} & -x_{1,3} + x_{2,2} & 1 & 0
        \end{bmatrix}.
    \end{equation*}
    For example, for the indecomposable Hessenberg function $h = (2,3,4,4)$, the corresponding Hessenberg patch ideal corresponding to $w_0$ and $h$ is given by
    \begin{equation*}
        I_{w_0, h} = \langle -x_{2,2}+x_{3,1}, \, -x_{1,2} + x_{1,3}(x_{2,2} - x_{3,1}) + x_{2,1}, \, -x_{1,3} + x_{2,2} \rangle.
    \end{equation*}
    \exampleqed
\end{example}

In work with Abe, DeDieu, and Galetto, the third author showed that there exists an explicit inductive formula for the polynomials $f^{w_0}_{k,\ell}$ \cite[Equation (3.6)]{ADGH}.
Building on this work, the second and third authors found an explicit formula for the initial term for the polynomials $f^{w_0}_{k,\ell}$ and recursive equations relating the $f^{w_0}_{k,\ell}$ for varying $k$ and $\ell$ \cite{DSH}. To proceed further, we require a monomial order, which in \cite{DSH} is defined as follows. 

\begin{definition}[{\cite[Definition 4.11]{DSH}}] \label{def:w0MonomialOrder}
    We denote by $<_n$ the lexicographical monomial order on $\mathbb C[\mathbf x_{w_0}]$ defined by $x_{i,j} >_n x_{i',j'}$ if $i < i'$, or, $i = i'$ and $j < j'$.
\end{definition}

Moreover, there exists an ordering of the polynomials $f^{w_0}_{k,\ell}$ such that the initial term of one polynomial does not appear in any of the polynomials later in the order.
This sequence can be obtained by reading from the matrix $(w_0M)^{-1}\mathsf N (w_0M)$ left-to-right along the bottom row, then left-to-right along the penultimate row, and so on.
Explicitly, this is the sequence
\begin{equation} \label{seq:f^w0s}
    f^{w_0}_{n,1}, \, f^{w_0}_{n,2}, \, \ldots, \, f^{w_0}_{n,n-1}, \, f^{w_0}_{n-1,1}, \, f^{w_0}_{n-1,2}, \, \ldots.
\end{equation}
It is shown in \cite{DSH} that the initial terms of the polynomials $f^{w_0}_{k,\ell}$ with respect to $<_n$ form a list of distinct indeterminates.
We have the following. 

\begin{lemma}[{\cite[Lemma 4.13]{DSH}}] \label{lem:f^w0-init}
    Let $k, \ell \in [n]$ satisfy $k > \ell + 1$.
    Then the following hold. 
    \begin{enumerate}
        \item \label{lem-part:f^w0-init}
        With respect to the monomial order $<_n$ of Definition \ref{def:w0MonomialOrder}, $\init_{<_n}(f^{w_0}_{k,\ell}) = -x_{n+1-k, \ell + 1}$.
        In particular, the initial term is an indeterminate (up to sign). 
        
        \item \label{lem-part:f^w0-init-seq}
        The initial term $-x_{n+1-k,\ell+1}$ of $f^{w_0}_{k,\ell}$ does not appear in any $f^{w_0}_{k',\ell'}$ which are listed after $f^{w_0}_{k,\ell}$ in the sequence \eqref{seq:f^w0s}.

        \item \label{lem-part:f^w0-init-indet}
        The indeterminate $x_{n+1-k,\ell+1}$ appears exactly once in $f^{w_0}_{k,\ell}$ and all other indeterminates $x_{i,j}$ appearing in $f^{w_0}_{k,\ell}$ either satisfy $i > n+1-k$ or $i = n+1-k$ and $j > \ell + 1$.
    \end{enumerate}
\end{lemma}

Let $h$ be an indecomposable Hessenberg function.
By Definition \ref{def:HessPatchIdealAndf^w}, the ideal $I_{w_0,h}$ is generated by the polynomials $f^{w_0}_{k,\ell}$ for all $k, \ell \in [n]$ with $k > h(\ell)$. By definition of indecomposable Hessenberg functions, if $k>h(\ell)$ then $k \geq \ell+1$, so the above lemma applies. We additionally have the following useful result from \cite{DSH}. 

\begin{proposition}[{\cite[Corollary 4.15]{DSH}}] \label{prop:w0Grobner}
    For any indecomposable Hessenberg function $h$, the generators $f^{w_0}_{k,\ell}$ form a Gr\"obner basis for the Hessenberg patch ideal $I_{w_0,h}$ with respect to $<_n$.
\end{proposition}

In this paper, we generalize Proposition \ref{prop:w0Grobner} to all regular nilpotent Hessenberg Schubert cells.

\begin{remark}
    Throughout this manuscript, we restrict to indecomposable Hessenberg functions.
    This restriction is quite natural since, in type $A$, any regular nilpotent Hessenberg variety is the product of indecomposable regular nilpotent Hessenberg varieties \cite[Theorem 4.5]{Drellich}.
\end{remark}

\subsection{$\mathsf S$-actions on regular nilpotent Hessenberg varieties} \label{subsec:S-action}

In this manuscript we will use a certain circle action on $\mathrm{Hess}(\mathsf{N},h)$ which allows us to restrict attention to an appropriate subset of Schubert cells for our analysis. It is the same circle action as is used in \cite{DSH}, so we keep the exposition brief and refer the reader to \cite[Section 2]{DSH} for details. 

Let $\mathsf{S} \cong \mathbb{C}^*$ denote the circle subgroup of the maximal torus of $GL_n(\mathbb C)$ given by 
\[
\mathsf{S} \coloneqq \{ \underline{t}\coloneqq \mathrm{diag}(t, t^2, \ldots, t^n) \, \mid \, t \in \mathbb{C}^*\}. 
\]
It is straightforward to check that $\mathsf{S}$ preserves $\mathrm{Hess}(\mathsf{N},h)$.
Moreover, since $\mathsf{S}$ is a subgroup of the usual maximal torus of $GL_n(\mathbb{C})$ and the maximal torus preserves the coordinate patches $\mathcal{N}_w$, it follows that $\mathsf{S}$ also preserves the cell $\mathcal{N}_w$, hence also the intersections $\mathrm{Hess}(\mathsf{N},h) \cap \mathcal{N}_w$.  

Additionally, we can also concretely compute the action of $\mathsf{S}$ on any local coordinate patch $\mathcal{N}_{w}$.  Recall that the standard maximal torus action on $GL_n(\mathbb C)/B$ is given by left multiplication on left cosets; since $\mathsf{S}$ is a subgroup, it acts in the same way. More precisely, given a matrix $w M$ representing a flag $[w M] \in GL_n(\mathbb C)/B$, we have  
\[
\underline{t} \cdot [w M] = [\underline t(w M)]
\]
for any $\underline t \in \mathsf S$.
To express $[\underline t(wM)]$ in terms of the coordinate chart $\mathcal{N}_w$, 
we must now find a matrix $M'$ with entries described by~\eqref{eqn:wM} such that $\underline tw M = w M'$. The simple observation here is that left multiplication by the diagonal matrix $\underline{t}$ changes the entries in the $(w(j),j)$-th spots to be $t^{w(j)}$ instead of a $1$, so in order to return the matrix to the required form, it is necessary to multiply on the right by the diagonal element $\mathrm{diag}(t^{-w(1)}, t^{-w(2)}, \ldots, t^{-w(n)})$, i.e. the diagonal matrix whose $j$-th diagonal entry is $t^{-w(j)}$. From this simple observation it is not hard to see that the matrix $w M'$ satisfies 
\begin{equation*}
[wM']_{i,j} = 
\begin{cases} 
1  & \text{if } i=w(j), \\
0 & \text{if } j>w^{-1}(i),\\ 
t^{i-w(j)} x_{i,j} & \text{otherwise}
\end{cases}
\end{equation*}
where the $x_{i,j}$ are the entries of $wM$ as in~\eqref{eqn:wM}. 

The above-described torus action on the patch $\mathcal{N}_w$ induces an $\mathsf S$-grading on the coordinate ring $\mathbb{C}[\mathbf{x}_w]$, where 
the indeterminate $x_{i,j}$ has degree $w(j)-i$.
In Section~\ref{sec:GB-HessSchubertCells}, we will restrict this action further to the Schubert cell at $w$, but the principle remains the same.
In fact, in this setting, we will see that this $\mathbb Z$-grading is positive.

\section{Coordinates on regular nilpotent Hessenberg Schubert cells} \label{sec:GB-HessSchubertCells}

In this section, we will relate coordinates on the regular nilpotent Hessenberg Schubert cells to certain coordinates on the $w_0$-chart of the same Hessenberg variety. This will allow us, in the next section, to adapt results about the $w_0$-chart from \cite{DSH} to yield results for the regular nilpotent Hessenberg Schubert cells, allowing us to draw conclusions about, e.g., the initial terms, recurrence relations, and Gr\"obner bases. Indeed, finding this explicit relation with the $w_0$-chart allows us to avoid reproving the technical details of \cite{DSH}. 

The Schubert cell of a permutation $w$ is defined by \[X^w_\circ\coloneqq BwB/B \subseteq GL_n(\mathbb{C})/B \cong \mathrm{Flags}(\mathbb{C}^n).\]
By the well-known theory of Bruhat decomposition, the Schubert cells are disjoint and their union is  $GL_n(\mathbb{C})/B$ (see \cite[\textsection 23.4]{FH} for further background on the Bruhat decomposition). Furthermore, each $X^w_\circ$ is isomorphic to a complex affine space of dimension $\ell(w)$, where $\ell(w)$ denotes the length of the permutation $w$. In fact, in a manner similar to how we described elements in $GL_n(\mathbb{C})/B$ in the previous section, we can parametrize the Schubert cell by viewing an element of $X^w_\circ$ as represented by a matrix of the form $\Omega_w$, where the $(i,j)$-th entry is given by:
\begin{equation} \label{eqn:Omega_ij}
    [\Omega_w]_{i,j} = 
    \begin{cases} 
        1 & \text{if } i = w(j), \\ 
        0 & \text{if } i > w(j) \text{ or } j > w^{-1}(i), \\
        z_{i,j} & \text{otherwise (i.e., } i < w(j) \text{ and } j < w^{-1}(i)). 
    \end{cases} 
\end{equation}
Note that the condition $i > w(j)$ corresponds to the matrix entries occurring below the $1$'s and the condition $j > w^{-1}(i)$ corresponds to the entries occurring to the right of the $1$'s. This parametrization illustrates that $X^w_\circ$ is isomorphic to an affine space, and we can view $\mathbf{z}_w := \{z_{i,j} \, \mid \, i<w(j) \, \textup{ and } \, j < w^{-1}(i)\}$ as indeterminates corresponding the affine coordinates on $X^w_\circ$. In particular, the affine coordinate ring of $X^w_\circ$ can be identified with the polynomial ring $\mathbb C[\mathbf z_w]$. 
Note that, in the discussion below, we will reserve the coordinates $x_{i,j}$ for entries of $wM$ as in \eqref{eqn:wM}, and the coordinates $z_{i,j}$ for entries of $\Omega_w$ as in~\eqref{eqn:Omega_ij}.

We now wish to relate the coordinates $\mathbf{z}_w$ with the coordinates $\mathbf{x}_{w_0}$. To do this, 
our first observation is that we can write the $z_{i,j}$ coordinates of the Schubert cell as a specialization of the $x_{i,j}$ coordinates of the $w_0$-chart (by setting certain variables to zero after an appropriate relabelling).
We start with a motivational example to demonstrate this behaviour, and describe this process more precisely afterwards.

\begin{example} \label{eg:motivation}
Let $n=4$ and $w=3421$. As we see below, for an element $w_0M\in\mathcal{N}_{w_0}$, there exists a permutation matrix $v_w$ such that if we right-multiply $w_0M$ by $v_w$ and then set some variables equal to $0$, then we obtain an element of the cell $X^w_\circ$. Indeed, in our example, $v_w$ turns out to be the permutation matrix 
$$
v_w = 
\begin{bmatrix}
    0 & 1 & 0 & 0 \\
    1 & 0 & 0 & 0 \\
    0 & 0 & 1 & 0 \\ 
    0 & 0 & 0 & 1 
\end{bmatrix}
$$
which corresponds to the permutation $v_w = 2134$. 
We can compute the result of right-multiplication by $v_w$ and setting $x_{3,1}$ to $0$, and obtain 
\begin{equation} \label{eqn:w0Mv_w}
    \left.\left(
    \begin{bmatrix}
        x_{1,1}  &  x_{1,2} & x_{1,3} & 1 \\
        x_{2,1}  &  x_{2,2} & 1 & 0 \\  
        x_{3,1}  &  1 & 0 & 0 \\ 
        1 & 0 & 0 & 0
    \end{bmatrix}
    \begin{bmatrix}
        0  & 1 & 0 & 0 \\
        1  &  0 & 0 & 0 \\  
        0  &  0 & 1 & 0 \\ 
        0 & 0 & 0 & 1   
    \end{bmatrix} 
    \right)\right|_{x_{3,1}=0}
    = 
    \begin{bmatrix}
        x_{1,2}  &  x_{1,1} & x_{1,3} & 1 \\
        x_{2,2}  &  x_{2,1} & 1 & 0 \\  
        1  &  0 & 0 & 0 \\ 
        0 & 1 & 0 & 0
    \end{bmatrix} \in X^w_\circ.
\end{equation}
For an element of the cell $X^w_\circ$, we may represent it as a matrix with entries as in \eqref{eqn:Omega_ij} by 
\begin{equation} \label{eqn:Omegaw in example}
\Omega_w=
 \begin{bmatrix}
    z_{1,1}  &  z_{1,2} & z_{1,3} & 1 \\
    z_{2,1}  &  z_{2,2} & 1 & 0 \\  
    1  &  0 & 0 & 0 \\ 
    0 & 1 & 0 & 0
\end{bmatrix}.
\end{equation}
Using the correspondence~\eqref{eqn:w0Mv_w}, we can identify the coordinates of $\Omega_w$ with those of $w_0M$ by a ring homomorphism $\psi_w$ from $\mathbb{C}[\mathbf x_{w_0}] \rightarrow \mathbb{C}[\mathbf z_w]$ defined by comparing the matrix of \eqref{eqn:w0Mv_w} entry-by-entry to $\Omega_w$ in~\eqref{eqn:Omegaw in example}. For instance, in this example, $z_{1,1}$ is identified with $x_{1,2}$, $z_{1,2}$ with $x_{1,1}$, and so on. 
The reader may verify that this is exactly the correspondence given in Definition \ref{def:psi} below, which formalizes and generalizes the phenomenon seen in this example. 
\exampleqed
\end{example}

Following the example, given a permutation $w \in S_n$, define $v_w \coloneqq w_0w$. Denote the $(i,j)$-th entry of the matrix $w_0M$ by $x_{i,j}$ for $i,j \in [n]$ with $i+j\leq n+1$. (Here we set $x_{i,n-i+1} \coloneqq 1$ for all $1 \leq i \leq n$, i.e., the anti-diagonal entries are set to be equal to $1$.) Under right-multiplication by $v_w$ (and by our convention~\eqref{eqn:permutationConvention} for permutation matrices), the columns of $w_0 M$ are permuted, and in particular, the $(i,j)$-th entry of the matrix $(w_0M)v_w$ is the $(i, v_w(j))$-th entry of the matrix $w_0M$.
Equivalently, we have
\begin{equation*} 
    [(w_0M)v_w]_{i, j} = [w_0M]_{i, v_w(j)}.
\end{equation*}
Moreover, since it is precisely the anti-diagonal entries of $w_0M$ which are equal to $1$, the above equality implies that the $(i,j)$-th entry of $(w_0M)v_w$ is equal to $1$ precisely when $i = n+1 - v_w(j) = w_0(v_w(j)) = w(j)$, where the last equality is by our choice of $v_w$. 
This means that the $1$'s in the matrix $(w_0M)v_w$ are located precisely at the same locations as the $1$'s in the permutation matrix corresponding to $w$. Therefore, if 
the variables appearing in $(w_0M)v_w$  below the 1's and to the right of the $1$'s are all equal to $0$, then the matrix represents an element of the cell $X^w_\circ$.

Indeed, from~\eqref{eqn:Omega_ij}, it then follows that in order for the matrix $(w_0M)v_w$ to be in $X^w_\circ$, we must have that $[(w_0M)v_w]_{i,j}=0$ if either 
\begin{equation*}
w(j) < i
\end{equation*} 
or 
\begin{equation*}
j>w^{-1}(i). 
\end{equation*} 
We can then ask which indeterminates appearing as entries in $w_0M$, obtained by right-multiplication by $v_w$, have to be equal to $0$ in order for the corresponding matrix $(w_0M)v_w$ to be contained in $X^w_\circ$. Indeed, let $x_{a,b}$ denote the $(a,b)$-th entry of $w_0M$ for $a+b < n+1$. (As we already saw, if $a+b=n+1$ then $x_{a,b}=1$ so we do not consider this case.) Based on the above paragraphs, it is straightforward now to see that if $(w_0M) v_w$ is in $X^w_\circ$ it is necessary to have $x_{a,b}=0$ if either 
\begin{equation} \label{ineq:below1}
    w(v_w^{-1}(b)) < a
\end{equation}
or 
\begin{equation} \label{ineq:rightOf1}
    v_w^{-1}(b) > w^{-1}(a).
\end{equation}
Since $v_w = w_0 w$ by definition, the inequality \eqref{ineq:below1} simplifies to $w_0(b) < a$ or equivalently $n + 1 - b < a$, which can never happen since we started with $a + b < n+1$. Therefore, the only non-vacuous condition is~\eqref{ineq:rightOf1}. 

The preceding discussion motivates the following definitions. (In the previous paragraph, we used indices $a$ and $b$ to distinguish the indices used for $w_0M$ versus $(w_0M)v_w$. We now revert to using $i,j$ to denote indices for the matrix $w_0M$.) Let 
$D_w$ denote the collection of indeterminates $x_{i,j}$ with $(i,j)$ satisfying~\eqref{ineq:rightOf1}; more precisely, let
\begin{equation} \label{eqn: def Dw} 
D_w \coloneqq \{ x_{i, j} \in \mathbf x_{w_0} \mid i+j \leq n \text{ and } v_w^{-1}(j) > w^{-1}(i) \}. 
\end{equation} 

\begin{definition} \label{def:psi}
    Define a homomorphism of rings $\psi_w: \mathbb C[\mathbf x_{w_0}] \to \mathbb C[\mathbf z_w]$ by \[ \psi_w(x_{i,j}) \coloneqq 
    \begin{cases} 
        0 & \text{if } x_{i,j} \in D_w, \\
        z_{i,v_w^{-1}(j)} & \text{if } x_{i, j} \notin D_w. \\
    \end{cases} \]
\end{definition}

\begin{remark} 
Geometrically, the ring homomorphism $\psi_w$ corresponds to the embedding of the Schubert cell $X^w_\circ$ into $\mathcal{N}_{w_0}$, given by right multiplication of an element $\Omega_w \in X^w_\circ$ by $v_w^{-1}$. Then $\Omega_w v_w^{-1} \in \mathcal{N}_{w_0}$, and certain coordinates are equal to $0$ by the discussion above, so the corresponding map on coordinate rings sets those coordinates to $0$.
\end{remark} 

In what follows, if $A$ is a matrix whose entries are in $\mathbb C[\mathbf x_{w_0}]$, we denote by $\psi_w(A)$ the matrix obtained by applying the ring homomorphism $\psi_w$ to each matrix entry of $A$, i.e., if $a_{i,j}$ is the $(i,j)$-th entry of $A$, then $\psi_w(a_{i,j})$ is the $(i,j)$-th entry of $\psi_w(A)$. 
By construction, this results in a matrix with entries in $\mathbb{C}[\mathbf{z}_w]$.  With this notation in place, we have the following observations.

\begin{remark} \label{rmk: omega_w}
    For any $w \in S_n$, we have that 
    $\Omega_w = \psi_w((w_0M)v_w)$ by construction of $\psi_w$.
\end{remark}

We also have the following. 

\begin{lemma} \label{lem:psiCommutesInverse}
    For any $w \in S_n$, $\left( \psi_w((w_0M)v_w) \right)^{-1} = \psi_w \left( ((w_0M)v_w)^{-1} \right)$.
\end{lemma}

\begin{proof}
    This follows from the facts that $\psi_w$ is a ring homomorphism applied entry-by-entry and that any entry of the inverse of a matrix can be expressed as a polynomial in terms of the entries of the given matrix. 
    Indeed, the determinant of $M$ is equal to $1$ and determinants of permutation matrices are $\pm 1$. 
    Hence the determinant of $(w_0M)v_w$ is $\pm 1$ and so the determinant in the denominator of the usual formula for a matrix inverse does not appear in this computation, and all entries of the matrix inverse are polynomials and not rational functions.
\end{proof}

\begin{definition} \label{def:HessSchubertCell}
    Let $w \in S_n$.
    As in Definition \ref{def:HessPatchIdealAndf^w}, define polynomials $g^w_{k,\ell}$ in $\mathbb C[\mathbf z_w]$ for $k, \ell \in [n]$ by 
    \begin{equation*} 
        g^w_{k,\ell} \coloneqq \left[ \Omega_w^{-1} \mathsf N \Omega_w \right]_{k,\ell}.
    \end{equation*}
    For an indecomposable Hessenberg function $h$, we define the {\bf regular nilpotent Hessenberg Schubert cell ideal corresponding to $w$ and $h$} to be the ideal 
    \begin{equation} \label{def:CellIdeal}
        J_{w,h} \coloneqq \langle g^w_{k,\ell} \mid k > h(\ell) \rangle \subseteq \mathbb C[\mathbf z_w].
    \end{equation}
\end{definition}

As in the case of the generators $f^w_{k,\ell}$ in Definition \ref{def:HessPatchIdealAndf^w}, it is straightforward to see from~\eqref{eq: matrix eq for Hess} that the $g^w_{k,\ell}$ set-theoretically cut out the intersection $\mathrm{Hess}(\mathsf{N},h) \cap X^w_\circ$ from $X^w_\circ$, justifying the terminology in Definition~\ref{def:HessSchubertCell} above. It will also follow from our algebraic arguments below that $J_{w,h}$ is radical (see Corollary~\ref{cor:radical cell ideal}), and hence $J_{w,h}$ also scheme-theoretically defines $\mathrm{Hess}(\mathsf{N},h) \cap X^w_\circ$. 

\begin{notation} \label{notation: lambda h}
It will also be useful in what follows to have notation for the number of generators $g^w_{k,\ell}$ of $J_{w,h}$ as listed in~\eqref{def:CellIdeal}. Let $\lambda_h$ denote the partition $(n-h(1), n-h(2),\ldots, n-h(n-1),n-h(n)=0)$. Then $\lvert \lambda_h \rvert$, the sum of the parts of $\lambda_h$, is $\lvert \lambda_h \rvert = n^2 - \sum_{i=1}^n h(i)$. This is equal to the number of generators $g^w_{k,\ell}$ as listed in~\eqref{def:CellIdeal}. 
\end{notation}

\begin{example} \label{eg:motivation-part2}
Continuing Example \ref{eg:motivation},
we now note that we can express the generators for the regular nilpotent Hessenberg Schubert cell ideal in terms of the Hessenberg $w_0$-patch ideal via the map $\psi_w$.
First, notice that 
\begin{equation} \label{eq: conj Omegaw} 
\Omega_w^{-1} \mathsf N \Omega_w = \begin{bmatrix}
    0 & 1 & 0 & 0 \\
    0 & 0 & 0 & 0 \\
    1 & -z_{2,1} & 0 & 0 \\
    -z_{1,3} + z_{2,1} & -z_{1,1} + z_{1,3}z_{2,1} + z_{2,2} & 1 & 0
\end{bmatrix}. 
\end{equation}
Meanwhile,
\[F \coloneqq (w_0M)^{-1}\mathsf{N}(w_0M) =  \begin{bmatrix}
    0  &  0 & 0 & 0 \\
    1  &  0 & 0 & 0 \\  
    -x_{2,2}+x_{3,1}  &  1 & 0 & 0 \\ 
    -x_{1,2} + x_{1,3}(x_{2,2} - x_{3,1}) + x_{2,1} & -x_{1,3}+x_{2,2} & 1 & 0
\end{bmatrix}\]
and
\begin{equation} \label{eq: conj vw} v_w^{-1}Fv_w = \begin{bmatrix}
    0  &  1 & 0 & 0 \\
    0  &  0 & 0 & 0 \\  
    1  &  -x_{2,2}+x_{3,1} & 0 & 0 \\ 
    -x_{1,3} +x_{2,2} & -x_{1,2} + x_{1,3}(x_{2,2} - x_{3,1}) + x_{2,1} & 1 & 0
\end{bmatrix}.
\end{equation} 
Since $\Omega_w = \psi_w((w_0M)v_w)$ as observed in Remark~\ref{rmk: omega_w}, it follows by Lemma \ref{lem:psiCommutesInverse} that $\Omega_w^{-1} = \psi_w(v_w^{-1}(w_0M)^{-1})$.
Applying $\psi_w$ entry-by-entry, it follows that $\psi_w(v_w^{-1}F v_w) = \Omega_w^{-1} \mathsf N \Omega_w$.
In particular, the generators satisfy $g^w_{k,\ell} = \psi_w( [v_w^{-1}F v_w]_{k,\ell} )$. For example, comparing the $(3,2)$-th entries of the RHS of~\eqref{eq: conj Omegaw} and~\eqref{eq: conj vw}, we can see that 
\begin{equation*}
    \begin{split} 
    \psi_w(-x_{2,2}+x_{3,1}) 
& = \psi_w(-x_{2,2}) + \psi_w(x_{3,1}) \\
& = - z_{2,1} + 0 = - z_{2,1} 
    \end{split} 
\end{equation*}
as claimed. We will formalize this computation in Lemma~\ref{lem:psiImage} below. 
\exampleqed
\end{example}

We have just seen in the above example that the generators $g^w_{k,\ell}$ in this special case have a simple algebraic relation to certain of the $f^{w_0}_{i,j}$ generators from Definition~\ref{def:HessPatchIdealAndf^w}, since the matrix entries of $F$ from the above example are precisely these generators. 
In the discussion that follows, the overarching theme is that we can use this correspondence to effectively translate results on the $w_0$-patch as obtained in \cite{DSH} to the regular nilpotent Hessenberg Schubert cells $\mathrm{Hess}(\mathsf{N},h) \cap X^w_\circ$. 
Before proceeding further into the details of the translation, however, we need first to make precise the permutations $w$ for which this translation makes sense. As we will show below, it turns out that the only Schubert cells $X^w_\circ$ for which the intersection $\mathrm{Hess}(\mathsf{N}, h) \cap X^w_\circ$ is non-empty (and hence the questions about local defining ideals is non-vacuous) are the permutations $w$ lying in the $\mathsf{S}$-fixed point set of $\mathrm{Hess}(\mathsf{N},h)$. The next lemma makes this precise. 

\begin{lemma} \label{lem:whenNonEmpty}
Let $w \in S_n$. Let $\mathsf{S}$ denote the copy of $\mathbb{C}^*$ acting on $\mathrm{Hess}(\mathsf{N},h)$ as described in Section~\ref{subsec:S-action}. Then $\mathrm{Hess}(\mathsf{N},h) \cap X^w_\circ \neq \emptyset$ if and only if $w \in \mathrm{Hess}(\mathsf{N},h)^{\mathsf{S}}$. 
\end{lemma}

\begin{proof} 
Since $w \in X^w_\circ$ for any $w \in S_n$, one direction is clear. Indeed, since $w\in X^w_\circ$, if $w \in \mathrm{Hess}(\mathsf{N},h)$ then $\mathrm{Hess}(\mathsf{N},h) \cap X^w_\circ \neq \emptyset$. 

It remains to show the other direction. 
Suppose there exists a point $gB \in \mathrm{Hess}(\mathsf{N},h) \cap X^w_\circ$, so it is non-empty. 
Following our notational convention in~\eqref{eqn:Omega_ij} for points in the Schubert cell $X^w_\circ$, we may represent this point $gB$ as a matrix $\Omega_w$ whose $(i,j)$-th entry is $1$ if $i=w(j)$, $0$ if $i>w(j)$ or $j > w^{-1}(i)$, and denoted $z_{i,j}$ otherwise.
We already observed in Section~\ref{subsec:S-action} that $\mathrm{Hess}(\mathsf{N},h)$ is preserved under the $\mathsf{S}$-action for our choice of $\mathsf{S}$. 
Schubert cells are also preserved, since they are preserved by the maximal torus action, and $\mathsf{S}$ is a subgroup of the maximal torus. 
Thus $\mathsf{S}$ preserves $\mathrm{Hess}(\mathsf{N},h) \cap \Omega_w$, and we conclude that the $\mathsf{S}$-orbit $\mathsf{S} \cdot gB$ is also contained in $\mathrm{Hess}(\mathsf{N},h) \cap \Omega_w$. 
Using the matrix representation of the point $gB$ and the definition of $\mathsf{S}$, it is straightforward to compute, following the outline given in Section~\ref{subsec:S-action}, that for an element $\underline{t} \coloneqq \mathrm{diag}(t, t^2, \ldots, t^n)$ for $t \in \mathbb{C}^*$, the element $\underline{t} \cdot gB$ can be represented by a matrix in $GL_n(\mathbb{C})$ whose $(i,j)$-th entry is $0$ if $i>w(j)$ or $j> w^{-1}(i)$, is $t^i$ if $i=w(j)$, and $t^i z_{i,j}$ otherwise. 
Such a matrix is not in the form given in~\eqref{eqn:Omega_ij} so we adjust by multiplication on the right by a torus element in $B$ given by $\mathrm{diag}(t^{-w(1)}, t^{-w(2)}, \ldots, t^{-w(n)})$, which then yields a matrix specifying the same flag, the entries of which are $0$ if $i>w(j)$ or $j > w^{-1}(i)$, $1$ if $i=w(j)$, and $t^{i-w(j)}z_{i,j}$ otherwise. 
Now we observe that by assumption, if the $(i,j)$-th entry is $t^{i-w(j)}z_{i,j}$, then $i<w(j)$ and therefore $i-w(j)<0$. 
This implies that as we take the limit as $t \to \infty$, the limit point of this orbit is precisely the permutation flag $w$. Since $\mathrm{Hess}(\mathsf{N},h)$ is closed, the limit point $w$ is also contained in $\mathrm {Hess}(\mathsf{N},h)$, and it is evident that $w$ is $\mathsf{S}$-fixed by the above computation. 
Hence $w \in \mathrm{Hess}(\mathsf{N},h)^{\mathsf{S}}$, as desired. 
\end{proof} 

The above lemma shows that we can restrict attention to permutations $w$ contained in $\mathrm{Hess}(\mathsf{N},h)^{\mathsf{S}}$. 
The third author and Abe, Horiguchi, and Masuda gave an explicit characterization of this set, as follows. 

\begin{lemma}[{\cite[Lemma 2.3]{AHHM}}] \label{lem:fixedPoint}
    Let $h$ be any Hessenberg function.
    Then, \[ {\rm Hess}(\mathsf N, h)^{\mathsf S} = \{ w \in S_n \mid w^{-1}(w(j)-1) \le h(j) \text{ for all } j \in [n] \}, \] where we take the convention that $w(0) = 0$ for all $w \in S_n$.
\end{lemma}

We can now relate the $f^{w_0}_{i,j}$ polynomials on the $w_0$-chart to the $g^w_{k,\ell}$ polynomials on the regular nilpotent Hessenberg Schubert cell. The following is a straightforward computation. 

\begin{lemma} \label{lem:psiImage}
    For any $w \in S_n$ and any $k, \ell \in [n]$, we have $g^w_{k, \ell} = \psi_w(f^{w_0}_{v_w(k), v_w(\ell)})$.
\end{lemma}

\begin{proof}
    By Remark \ref{rmk: omega_w}, Lemma~\ref{lem:psiCommutesInverse}, and  Definition~\ref{def:HessSchubertCell} of the  $g^w_{k,\ell}$, we can compute that 
    \begin{align*}
        g_{k,\ell}^w &= [\Omega_w^{-1}\mathsf N \Omega_w]_{k,\ell} = [(\psi_w((w_0M)v_w))^{-1}\mathsf N(\psi_w((w_0M)v_w))]_{k,\ell} \\
        &= [\psi_w(((w_0M)v_w)^{-1})\mathsf N(\psi_w((w_0M)v_w))]_{k,\ell} \\
        \intertext{and since $\psi_w$ is a ring homomorphism,}
        &= [\psi_w((v_w^{-1}(w_0M)^{-1})\mathsf N(w_0M)v_w)]_{k,\ell} \\ 
        &= \psi_w([v_w^{-1}(w_0M)^{-1}\mathsf N(w_0M)v_w]_{k,\ell})
    \end{align*}
    where the last equality holds since $\psi_w$ is applied to the matrix entry-by-entry.
    The claim of the lemma then follows from the facts that multiplication of $(w_0M)^{-1}\mathsf N(w_0M)$ on the left by $v_w^{-1}$ permutes rows and multiplication on the right by $v_w$ permutes columns.
\end{proof}

Although the above result holds for any $k, \ell \in [n]$, in the next sections we will always impose further restrictions on $(k, \ell)$ so that they correspond to the generators $g^w_{k,\ell}$ for the ideal $J_{w, h}$.

\section{A Gr\"obner basis for defining ideals of regular nilpotent Hessenberg Schubert cells} \label{sec: Grobner for Hess Schubert cells}

In the last section, we provided an explicit relationship between the generators $g^w_{k,\ell}$ of the local defining ideal of regular nilpotent Hessenberg Schubert cells and their corresponding entries $f^{w_0}_{v_w(k), v_w(\ell)}$ in the $w_0$-chart. The main purpose of the present section is to show that these generators in fact form a Gr\"obner basis for the ideal $J_{w,h}$. The main strategy is to concretely describe the relationship between the initial terms of the $g^w_{k,\ell}$ and those of the corresponding $f^{w_0}_{v_w(k), v_w(\ell)}$ through the ring homomorphism $\psi_w$. 

We first note that if a polynomial $f^{w_0}_{a,b}$ is equal to either $0$ or $1$, which occurs when $a \leq b+1$, then $\psi_w$ will map $f^{w_0}_{a,b}$ to itself and the initial terms are vacuously preserved. 
From Lemma \ref{lem:f^w0-init}\eqref{lem-part:f^w0-init} it follows that the case in which $f^{w_0}_{a,b}$ has a non-constant initial term is when $a > b+1$. Therefore, in comparing initial terms of $g^{w}_{k,\ell}$ with the corresponding $f^{w_0}_{v_w(k),v_k(\ell)}$ via Lemma~\ref{lem:psiImage}, we need only consider the case when $v_w(k) > v_w(\ell) + 1$.

In order to discuss Gr\"obner bases, we first need to fix a monomial order. Our first step is therefore to adapt the monomial order $<_n$ from Definition \ref{def:w0MonomialOrder} to a monomial order on $\mathbb C[\mathbf z_w]$. We do this using the permutation $v_w$.

\begin{definition} \label{def:cellMonomialOrder}
    Define a lexicographic monomial order $<_n^w$ on $\mathbb C[\mathbf z_w]$ by $z_{i,j} >_n^w z_{i',j'}$ if $i < i'$ or $i=i'$ and $v_w(j) < v_w(j')$.
\end{definition}

\begin{remark}
    In the case that $w = w_0$, the permutation $v_w$ is the identity permutation, so we recover from $<_n^w$ exactly the monomial order $<_n$ from Definition \ref{def:w0MonomialOrder}.
\end{remark}

\begin{lemma} \label{lem:nonzeroPsi}
Let $w \in \mathrm{Hess}(\mathsf{N},h)^{\mathsf{S}}$. Assume $k,\ell\in [n]$ satisfy $k \geq h(\ell)$ and $v_w(k) > v_w(\ell)+1$. Then, 
\begin{enumerate} 
\item The initial term of $f^{w_0}_{v_w(k), v_w(\ell)}$ does not get mapped to zero under $\psi_w$.
\item \label{lem-part:psi injective} Furthermore, if $(k',\ell')$ also satisfies the hypotheses and $(k',\ell') \neq (k, \ell)$,
    then $\psi_w (\init_{<_n} (f^{w_0}_{v_w(k'), v_w(\ell')}) ) \neq \psi_w ( \init_{<_n} (f^{w_0}_{v_w(k), v_w(\ell)} )).$
    That is, $\psi_w$ is injective when we restrict its domain to $\mathbb C[\mathbf x_{w_0} \setminus D_w]$.
\end{enumerate} 
\end{lemma}

\begin{proof}
     Suppose $k, \ell$ satisfy $k \geq h(\ell)$ and $v_w(\ell)+1 > v_w(k)$. 
        By Lemma \ref{lem:f^w0-init}\eqref{lem-part:f^w0-init-seq}, the initial term is the indeterminate $\init_{<_n}(f^{w_0}_{v_w(k),v_w(\ell)}) = -x_{n+1-v_w(k), v_w(\ell)+1}$.  To show that this does not map to $0$ under $\psi_w$, we will show that $x_{n+1-v_w(k), v_w(\ell)+1} \not \in D_w$. By~\eqref{eqn: def Dw}, this is equivalent to showing that 
        $$
        n+1-v_w(k)+v_w(\ell)+1 \leq n \quad \textup{and} \quad v_w^{-1}(v_w(\ell)+1) \leq  w^{-1}(n+1-v_w(k)).
        $$
        The first inequality is equivalent to $v_w(\ell) +2 \leq v_w(k)$, and this follows from the assumption on $k$ and $\ell$. 
        For the second inequality, observe that
        $$
        n+1-v_w(k) = w_0(v_w(k))
        $$
        and hence 
        $$
        w^{-1}(n+1-v_w(k)) = w^{-1}(w_0v_w(k)) = (w_0 v_w)^{-1}(w_0v_w(k)) = k.
        $$
        Hence, to prove the second inequality it suffices to show that 
        $$
        v_w^{-1}(v_w(\ell)+1) \leq k.
        $$
        Since $v_w = w_0 w$, we may also compute the left-hand side as 
        \begin{equation*} 
        \begin{split} 
    v_w^{-1}(v_w(\ell)+1) & = w^{-1}w_0(w_0w(\ell)+1) \\
    & = w^{-1}w_0(n+1-w(\ell)+1) \\
    & = w^{-1}(w(\ell)-1)
        \end{split} 
        \end{equation*} 
        so we need to prove 
        $$
        w^{-1}(w(\ell)-1) \leq k.
        $$
    Recall that $w\in \mathrm{Hess}(\mathsf{N},h)^{\mathsf{S}}$. By Lemma~\ref{lem:fixedPoint} we know therefore that 
    $$
    w^{-1}(w(\ell)-1) \leq h(\ell)
    $$
    and by our assumption $h(\ell) \leq k$, and hence we obtain the inequality 
    $$
    w^{-1}(w(\ell)-1) \leq k
    $$
    as desired. This proves the first claim. 

    Now if $\psi_w(x_{i,j}) = \psi_w(x_{i',j'})$ and are nonzero, by Definition \ref{def:psi} we must have $i = i'$ and $v_w^{-1}(j) = v_w^{-1}(j')$, which forces $j = j'$, so $x_{i,j} = x_{i',j'}$.
    That is, $\psi_w|_{\mathbb C[\mathbf x_{w_0}\setminus D_w]}$ is injective. This proves the second claim. 
\end{proof}

\begin{lemma} \label{lem:compatibleInit}
    Under the hypotheses of Lemma \ref{lem:psiImage}, $\psi_w$ respects the monomial orders $<_n^w$ and $<_n$.
    That is, $\init_{<_n^w}(\psi_w(f^{w_0}_{v_w(k),v_w(\ell)})) = \psi_w(\init_{<_n}(f^{w_0}_{v_w(k),v_w(\ell)}))$.
\end{lemma}

\begin{proof}
    By the result of Lemma \ref{lem:f^w0-init}\eqref{lem-part:f^w0-init-indet}, the initial term of $f^{w_0}_{v_w(k),v_w(\ell)}$ is an indeterminate that appears exactly once in $f^{w_0}_{v_w(k),v_w(\ell)}$ and by Lemma \ref{lem:nonzeroPsi}, this initial term does not get mapped to zero.
    As a result, it suffices to show that if $x_{i,j}, x_{i',j'} \notin D_w$ with $x_{i,j} >_n x_{i',j'}$, then $\psi_w(x_{i,j}) >_n^w \psi_w(x_{i',j'})$.

    So suppose that $x_{i,j}, x_{i',j'} \notin D_w$ with $x_{i,j} >_n x_{i',j'}$.
    By Definition \ref{def:w0MonomialOrder}, we have that either $i < i'$, or, $i = i'$ and $j < j'$.
    Then by Definition \ref{def:psi}, we have that $\psi_w(x_{i,j}) >_n^w \psi_w(x_{i',j'})$ if and only if $z_{i,v_w^{-1}(j)} >_n^w z_{i',v_w^{-1}(j')}$.
    The case $i < i'$ is immediate, so assume that $i = i'$.
    Then, the inequality holds if and only if $v_w(v_w^{-1}(j)) < v_w(v_w^{-1}(j'))$, i.e., $j < j'$.
\end{proof}

\begin{corollary} \label{cor: init Jwh} 
Let $J_{w,h}$ be the regular nilpotent Hessenberg Schubert cell ideal defining $\mathrm{Hess}(\mathsf{N},h) \cap X^w_\circ$ as a subset of $X^w_\circ$. Then, with respect to the monomial order $<^w_n$ defined above, the initial terms of the generators $g^w_{k,\ell}$ of $J_{w,h}$ are distinct indeterminates. 
\end{corollary} 

\begin{proof} 
 Let $\{g^w_{k_1,\ell_1}, \ldots, g^w_{k_r,\ell_r}\}$ be the set of generators of $J_{w,h}$ from Definition \ref{def:HessSchubertCell}.
    By Lemma \ref{lem:psiImage}, this is equal to the set $\{ \psi_w(f^{w_0}_{v_w(k_1),v_w(\ell_1)}), \ldots, \psi_w(f^{w_0}_{v_w(k_r),v_w(\ell_r)}) \}$.
    Then, by Lemma \ref{lem:compatibleInit}, $\init_{<_n^w}(g^w_{k_i,\ell_i}) = \psi_w(\init_{<_n}(f^{w_0}_{v_w(k_i),v_w(\ell_i)}))$ and by Lemma \ref{lem:nonzeroPsi}, the initial terms of $g^w_{k_i,\ell_i}$ with respect to $<_n^w$ are distinct indeterminates, as claimed. 
\end{proof} 

It is now straightforward to see that the generators $\{g_{k,\ell}^w\}$ of $J_{w, h}$ form a Gr\"obner basis with respect to $<_n^w$. 

\begin{theorem} \label{thm:subsetGB}
    For any indecomposable Hessenberg function $h$ and any $w \in {\rm Hess}(\mathsf N, h)^{\mathsf S}$, the generators $g_{k,\ell}^w$ for the regular nilpotent Hessenberg Schubert cell ideal $J_{w,h}$ form a Gr\"obner basis for $J_{w,h}$ with respect to $<_n^w$. Moreover, the initial ideal $\init_{<^w_n}(J_{w,h})$ is an ideal generated by distinct indeterminates. 
\end{theorem}

\begin{proof}
    We saw in Corollary~\ref{cor: init Jwh} that the initial terms of the generators are distinct indeterminates. 
    This implies that the $S$-polynomial for any pair of generators $g^w_{k,\ell}$ and $g^w_{k',\ell'}$ is given by $S(g^w_{k,\ell}, g^w_{k',\ell'}) = \init_{<_n^w}(g^w_{k',\ell'})g^w_{k,\ell} - \init_{<_n^w}(g^w_{k,\ell})g^w_{k',\ell'}$.
    Buchberger's Criterion \cite[Chapter 2.6, Theorem 6]{CLO} immediately implies that the generators $g^w_{k,\ell}$ form a Gr\"obner basis with respect to $<_n^w$. Since we saw in Corollary~\ref{cor: init Jwh} that the initial terms of the $g^w_{k,\ell}$ are distinct indeterminates, and since these form a Gr\"obner basis, the initial ideal is an ideal generated by distinct indeterminates as claimed. 
\end{proof}

We now provide a sketch of an alternate proof using the fact that the image of a Gr\"obner basis is again a Gr\"obner basis under a specialization map if the initial terms appearing do not map to zero.

\begin{proof}[Alternate proof sketch of Theorem \ref{thm:subsetGB}]
    Suppose that $\{g^w_{k_1,\ell_1}, \ldots, g^w_{k_r,\ell_r}\}$ are the generators for $J_{w,h}$ from Definition \ref{def:HessSchubertCell}.
    By Lemma \ref{lem:psiImage}, each $g^w_{k,\ell}$ is the image of a certain $f^{w_0}_{a,b}$ under the ring homomorphism $\psi_w$.
    Furthermore by Lemma \ref{lem:nonzeroPsi}, the initial terms of these $f^{w_0}_{a,b}$ with respect to $<_n$ are not mapped to $0$ by $\psi_w$. Let us define a new lexicographic monomial order $\prec^{w_0,w}_n$ on $\mathbb{C}[\mathbf x_{w_0}]$ which respects the ordering $<_n$ except for all $x_{i,j}\in D_w$ which receive the least weight.
    That is, if $x_{i,j} \in D_w$ and $x_{i',j'} \notin D_w$, then $x_{i,j} \prec_n^{w_0,w} x_{i',j'}$, and if $x_{i,j}, x_{i',j'} \in D_w$ or $x_{i,j}, x_{i',j'} \notin D_w$, then $x_{i,j} \prec_n^{w_0,w} x_{i',j'}$ if $x_{i,j} <_n x_{i',j'}$.
    
    Then $\{f^{w_0}_{v_w(k_1),v_w(\ell_1)}, \ldots, f^{w_0}_{v_w(k_r),v_w(\ell_r)} \}$ is a Gr\"obner basis for the ideal it generates with respect to $\prec^{w_0,w}_n$, since the initial terms of the $f^{w_0}_{a,b}$ with respect to $<_n$ are distinct indeterminates not in $D_w$. Then per \cite[Definition 4]{Gianni}, $\prec^{w_0,w}_n$ is a $\psi_w$-admissible order, so the result of \cite[Theorem 1]{Gianni}, together with Lemmas \ref{lem:nonzeroPsi} and \ref{lem:compatibleInit}, completes the proof.
    (Technically, $\psi_w$ is a composition of a specialization map with a change of variable names, but this is a small adjustment to the statement of \cite[Theorem 1]{Gianni}.) 
\end{proof}

From the above results, it follows immediately that the regular nilpotent Hessenberg Schubert cell ideal $J_{w, h}$ is radical and hence is  the defining ideal of ${\rm Hess}(\mathsf N, h) \cap X^w_\circ$.

\begin{corollary} \label{cor:radical cell ideal}
Let $w \in \mathrm{Hess}(\mathsf{N},h)^{\mathsf{S}}$.     
Then the ideal $J_{w, h}$ is radical.
\end{corollary}

\begin{proof}
Since $\init_{<_n^w}(J_{w, h})$ is an ideal of indeterminates, it is a radical ideal. Thus, so too is $J_{w,h}$ (e.g., see \cite[Exercise 4.2.16]{CLO}).
\end{proof}

\section{Applications of the Main Results} \label{sec:applications}

In this section, we discuss some consequences of our main results. The results of the previous sections imply that the regular nilpotent Hessenberg Schubert cells are affine spaces, as we show in Proposition~\ref{cor:PavingByAffines}. This allows us to obtain several immediate applications of our main results which connect to existing literature and illustrate the usefulness of our computational-algebraic perspective. 

The content of this section is as follows. First, in Section~\ref{subsec:complete intersections} we see that the local defining ideals of regular nilpotent Hessenberg Schubert cells are complete intersections. Second, in Section~\ref{sec: affine paving} we show -- as advertised above -- that the regular nilpotent Hessenberg varieties $\mathrm{Hess}(\mathsf{N},h)$ are paved by affine spaces. Third, in Section~\ref{subsec: hilbert series} we compute the Hilbert series of the ideals of $\mathrm{Hess}(\mathsf{N},h) \cap X^w_\circ$. The content of Section~\ref{subsec: applications GVD} is to show that these cell ideals $J_{w,h}$ are geometrically vertex decomposable. Finally, in Section~\ref{subsec: applications Frobenius} we state and prove some results on Frobenius splitting related to regular nilpotent Hessenberg Schubert cells.

\subsection{Complete intersections} \label{subsec:complete intersections}

As an immediate corollary of the results in Section \ref{sec:GB-HessSchubertCells}, we see that, in analogy with the regular nilpotent Hessenberg patch ideal at the $w_0$-chart (as seen in \cite{DSH}), regular nilpotent Hessenberg Schubert cell ideals are also complete intersection ideals. We begin with the following. 

\begin{lemma} \label{lem:triangular_order}
Let $\{g^w_{k_1,\ell_1}, \ldots, g^w_{k_r,\ell_r}\}$ be the set of nonzero generators of $J_{w,h}$ from Definition \ref{def:HessSchubertCell}, ordered so that $\init_{<_n^w}(g^w_{k_1,\ell_1}) >_n^w \ldots  >_n^w \init_{<_n^w}(g^w_{k_r,\ell_r})$. Then $\init_{<_n^w}(g^w_{k_i,\ell_i})$ does not divide any term of $g^w_{k_j,\ell_j}$ for any
$j>i$.
\end{lemma}
\begin{proof}
By Lemma \ref{lem:psiImage}, the ordered set of generators $\{g^w_{k_1,\ell_1}, \ldots, g^w_{k_r,\ell_r}\}$ is equal to the ordered set $\{ \psi_w(f^{w_0}_{v_w(k_1),v_w(\ell_1)}), \ldots, \psi_w(f^{w_0}_{v_w(k_r),v_w(\ell_r)}) \}$.  
Furthermore, we have that \[\init_{<_n}(f^{w_0}_{v_w(k_1),v_w(\ell_1)}) >_n \ldots  >_n \init_{<_n}(f^{w_0}_{v_w(k_r),v_w(\ell_r)})\]
and so $\init_{<_n}(f^{w_0}_{v_w(k_i),v_w(\ell_i)})$ does not divide any term of $f^{w_0}_{v_w(k_j),v_w(\ell_j)}$ for $j>i$ by \cite[Lemma 4.13]{DSH}. 
Applying $\psi_w$, we see that $\init_{<_n^w}(g^w_{k_i,\ell_i})$ does not divide any term of $g^w_{k_j,\ell_j}$ for 
any
$j>i$, 
completing the proof.
\end{proof} 

Before stating the main result of this subsection we include some notation for the height of each Hessenberg Schubert cell ideal.
As discussed at the beginning of Section \ref{sec: Grobner for Hess Schubert cells}, the nonzero generators for the Hessenberg Schubert cell ideal $J_{w,h}$ are the $g^w_{k, \ell}$ that satisfy both $k > h(\ell)$ and $v_w(k) > v_w(\ell) + 1$.
We denote by $\Lambda_{w,h}$ the number of nonzero generators $g^w_{k,\ell}$ for the ideal $J_{w,h}$.

We now recall from \cite{DSH} the definition of a triangular complete intersection and then state the main result of this subsection.

\begin{definition}[{\cite[Definition 3.3]{DSH}}]
    A complete intersection ideal $I \subseteq \mathbb C[x_1, \ldots, x_n]$ of height $r$ is called a {\bf triangular complete intersection} of height $r$ with respect to a monomial order $<$ if there exists an ordered set of generators $f_1, \ldots, f_r$ of $I$ satisfying:
    \begin{enumerate}
        \item the initial terms $\init_<(f_i)$ are unique indeterminates up to multiplication by units;

        \item $\init_<(f_i)$ does not divide any term of $f_j$ for any $j > i$.
    \end{enumerate}
\end{definition}

\begin{theorem} \label{thm: complete intersection}
For any indecomposable Hessenberg function $h$ and any $w \in {\rm Hess}(\mathsf N, h)^{\mathsf S}$, the regular nilpotent Hessenberg Schubert cell ideal $J_{w,h}$ is a complete intersection of  height $\Lambda_{w,h}$ and the nonzero $g_{k,\ell}^w$ defining $J_{w,h}$ form a minimal generating set. Furthermore, $J_{w,h}$ is a triangular complete intersection. 
\end{theorem}

\begin{proof}
    The result of Theorem \ref{thm:subsetGB} says that $\init_{<_n^w}(J_{w,h})$ is an ideal of indeterminates and hence is a complete intersection. Thus, by \cite[Corollary 19.3.8]{EGA}, $J_{w,h}$ is also a complete intersection. That the nonzero generators $g^w_{k, \ell}$ form a minimal generating set follows from Lemma \ref{lem:triangular_order}.

    To see that $J_{w,h}$ has height $\Lambda_{w, h}$, it suffices to show that $J_{w,h}$ is prime \cite[Theorem 1.8A(b)]{Hartshorne}.
    Relabel the variables in the ring as $z_1, \ldots, z_n$.
    By Lemma \ref{lem:triangular_order} we may reduce the generators $g^w_{k,\ell}$ to a reduced Gr\"obner basis for $J_{w,h}$ with respect to $<_n^w$ as $g_1, \ldots, g_{\Lambda_{w,h}}$ where $\init_{<_n^w}(g_i)$ is an indeterminate that does not divide any term of $g_j$ for any distinct $i$ and $j$.
    By a relabeling of the variables, we can write $g_i = z_i - q_i$ for each $1\leq i\leq \Lambda_{w,h}$ where $q_i$ is a polynomial in the variables $z_{\Lambda_{w,h}+1},\ldots,z_n$. It follows that 
    \[ 
    \mathbb C[\mathbf z_w]/J_{w,h} \cong \mathbb C[q_1, \ldots, q_{\Lambda_{w,h}}, z_{\Lambda_{w,h}+1}, \ldots, z_n] \cong \mathbb C[z_{\Lambda_{w,h}+1}, \ldots, z_n].
    \]
    Hence $J_{w,h}$ is prime and \cite[Theorem 1.8A(b)]{Hartshorne} implies that $J_{w,h}$ has height $\Lambda_{w,h}$, as desired.
    That $J_{w,h}$ is a triangular complete intersection now follows from the present argument, Corollary \ref{cor: init Jwh}, and Lemma \ref{lem:triangular_order}.
\end{proof}

In the above proof, we have also shown the following.

\begin{corollary}  \label{cor:CItoAffine}
    If $I \subseteq \mathbb C[x_1, \ldots, x_n]$ is a triangular complete intersection of height $r$, then $I$ is a prime ideal and $\mathbb V(I) \cong \mathbb A^{n - r}$.
\end{corollary}

\subsection{Affine pavings} \label{sec: affine paving}

An \textbf{affine paving} of an algebraic variety $X$ is an ordered partition of $X$ into disjoint $X_0, X_1, X_2,\dots$, such that
\begin{itemize}
    \item each $X_i$ is homeomorphic to an affine space; and
    \item each of the finite unions $\cup_{i=0}^r X_i$ is a Zariski closed subset of $X$.
\end{itemize}
Tymoczko proved in \cite{Tymoczko} that regular nilpotent Hessenberg varieties in the classical Lie type are paved by affine spaces, using delicate combinatorial and Lie-theoretic arguments. 
In this subsection, we provide a quick alternate proof of this result in Lie type $A$.

Combining Corollary \ref{cor:CItoAffine} with Theorem \ref{thm: complete intersection} immediately yields the following result.
Recall from the previous subsection that $\Lambda_{w,h}$ denotes the height of the regular nilpotent Hessenberg Schubert cell ideal $J_{w, h}$.

\begin{proposition} \label{cor:PavingByAffines}
    Let $h$ be an indecomposable Hessenberg function and $w\in \mathrm{Hess}(\mathsf N, h)^{\mathsf S}$. Let $r_w$ denote the dimension of the Schubert cell $X^w_\circ$. Then, the Hessenberg Schubert cell $\mathrm{Hess}(\mathsf N, h) \cap X_\circ^w$ is isomorphic to the affine space $\mathbb{A}^{r_w-\Lambda_{w,h}}$.
\end{proposition}

Next recall that the flag variety $GL_n(\mathbb{C})/B$ has an affine paving. The disjoint sets $X_0, X_1, X_2,\dots$ in the paving are the Schubert cells and they are ordered by any total order which refines Bruhat order. We will now see that this affine paving induces an affine paving on each regular nilpotent Hessenberg variety. 

\begin{theorem}
Let $h$ be an indecomposable Hessenberg function. Then ${\rm Hess}(\mathsf N, h)\subseteq GL_n(\mathbb{C})/B$ has an affine paving by the set of all ${\rm Hess}(\mathsf N, h)\cap X^w_\circ$, $w\in {\rm Hess}(\mathsf N, h)^{\mathsf S}$, totally ordered by any order which refines Bruhat order.
\end{theorem}

\begin{proof} 
By Lemma \ref{lem:whenNonEmpty}, $\text{Hess}(\mathsf N,h)\cap X^w_\circ$ is non-empty if and only if $w\in \text{Hess}(\mathsf N,h)^{\mathsf S}$. 
By Proposition \ref{cor:PavingByAffines}, each  non-empty $\text{Hess}(\mathsf N,h)\cap X^w_\circ$ is isomorphic to affine space. 
The result now follows by noting that the set of all Schubert cells form an affine paving of $GL_n(\mathbb{C})/B$ and the Hessenberg variety $\text{Hess}(\mathsf N,h)$ is a closed subvariety of $GL_n(\mathbb{C})/B$.
\end{proof}

\subsection{Hilbert series} \label{subsec: hilbert series}

In this section, we compute the Hilbert series of regular nilpotent Hessenberg Schubert cell ideals. The computation is a straightforward application of Theorem \ref{thm: complete intersection}, together with well-known techniques regarding Hilbert series of complete intersections and initial ideals. We start by quickly reviewing this background on Hilbert series.

Recall that the Hilbert series of a positively $\mathbb{Z}$-graded ring $R = \mathbb{C}\oplus R_1\oplus R_2\oplus \ldots$ is the series
\[ H_R(t) = \sum_{k=0}^\infty \dim_\mathbb{C}(R_k)t^k.\]
The following results are well-known but are stated here for completeness.

\begin{lemma}\cite[Theorem 15.26 and Exercise 21.17]{Eisenbud} \label{lem: Hilbert Series}
Suppose that $I$ is a homogeneous ideal of $R=\mathbb{C}[x_1,\ldots,x_n]$ with respect to a positive $\mathbb{Z}$-grading on $R$. Let $<$ be a monomial order on $R$. Then
\begin{enumerate}
    \item $H_{R/I}(t) = H_{R/\init_<(I)}(t)$.
    \item Let $I=\langle f_1,\ldots,f_k\rangle$ be a complete intersection of height $k$ where each $f_i$ is homogeneous and $\mathrm{deg}(f_i)=d_i$. Then \[ H_{R/I}(t) = H_R(t)\prod_{i=1}^k(1-t^{d_i}).\]
\end{enumerate}
\end{lemma}

Although the generators $g^w_{k, \ell}$ are not homogeneous with respect to the standard grading, we again use the $\psi_w$ map to define a \emph{non-standard} grading on $\mathbb C[\mathbf z_w]$. It will turn out, as we see below, that with respect to this non-standard grading, the $g^w_{k, \ell}$ are homogeneous.

In Section \ref{subsec:S-action} of this manuscript we defined, following \cite{DSH}, a $\mathbb Z$-grading on $\mathbb C[\mathbf x_w]$ by $\deg(x_{i, j}) := w(j) - i$. In the special case $w=w_0$, this becomes a positive $\mathbb Z$-grading on $\mathbb C[\mathbf x_{w_0}]$ given by $\deg(x_{i,j}) := w_0(j)-i$ (see Section \ref{subsec:S-action} and \cite[Lemma 2.18]{DSH}). We now define a $\mathbb Z$-grading on $\mathbb C[\mathbf z_w]$ by using this grading on $\mathbb C[\mathbf x_{w_0}]$ together with the algebra homomorphism $\psi_w$, namely, we define $\deg(z_{i, j}) \coloneqq \deg(\psi_w^{-1}(z_{i, j})) = w_0(v_w(j)) - i$, for all $z_{i, j} \in \mathbf z_w$.
That this is well-defined follows from Lemma \ref{lem:nonzeroPsi}\eqref{lem-part:psi injective}.

\begin{remark} \label{remark: two degrees agree}
    The grading on $\mathbb C[\mathbf z_w]$ given above agrees with the grading that arises from the $\mathsf S$-action on ${\rm Hess}(\mathsf N, h) \cap X^w_\circ$, defined in Section \ref{subsec:S-action}.
    Indeed, for any $z_{i, j} \in \mathbf z_w$, the grading from the $\mathsf S$-action yields $\deg(z_{i, j}) = w(j) - i$.
    By the construction of $v_w$, we know that  $w = w_0v_w$. From this it follows that the two definitions agree.
\end{remark}

\begin{example} 
    We continue Example \ref{eg:motivation-part2} and show that in this case, the two definitions of degree indeed agree, as pointed out in Remark~\ref{remark: two degrees agree}.
    The following table uses the definition arising from the $\mathsf S$-action on ${\rm Hess}(\mathsf N, h) \cap X^w_\circ$.
    \begin{center}
    \begin{tabular}{|c||c|c|c|c|c|} \hline 
        $z_{i, j}$ & $z_{1, 1}$ & $z_{1, 2}$ & $z_{1, 3}$ & $z_{2, 1}$ & $z_{2, 2}$
        \\ \hline 
        $\deg(z_{i, j}) = w(j) - i$ & 2 & 3 & 1 & 1 & 2
        \\ \hline
    \end{tabular} 
    \end{center}
    Meanwhile the following table uses the definition in terms of $\psi_w^{-1}$.
    Recall that $w_0(a) = n+1-a$ for all $a \in [n]$.
    \begin{center}
    \begin{tabular}{|c||c|c|c|c|c|} \hline 
        $z_{i, j}$ & $z_{1, 1}$ & $z_{1, 2}$ & $z_{1, 3}$ & $z_{2, 1}$ & $z_{2, 2}$
        \\ \hline 
        $\psi_w^{-1}(z_{i, j})$ & $x_{1, 2}$ & $x_{1, 1}$ & $x_{1, 3}$ & $x_{2, 2}$ & $x_{2, 1}$
        \\ \hline
        $\deg(z_{i, j}) = w_0(v_w(j)) - i$ & 2 & 3 & 1 & 1 & 2
        \\ \hline
    \end{tabular} 
    \end{center}
    Moreover, notice that the generators $g^w_{k,\ell}$ are homogeneous with respect to this grading.
    Indeed, $g^w_{4,1}$ has degree 1 and $g^w_{4, 2}$ has degree 2.
    \exampleqed
\end{example}

The following lemma shows that it is no accident that the generators $g^w_{k, \ell}$ in the previous example are homogeneous.
Indeed, the generators $f^{w_0}_{k,\ell}$ for the $w_0$-chart are homogeneous with respect to this grading, so our definition of the grading on $\mathbb C[\mathbf z_w]$ via the $\psi_w$ map allows us to easily translate the homogeneity to the ${\rm Hess}(\mathsf N, h)\cap X^w_\circ$ setting.

\begin{lemma} \label{lem:grading}
    Let $h$ be an indecomposable Hessenberg function and $w \in {\rm Hess}(\mathsf N, h)^{\mathsf S}$.
    Then, the above defines a positive $\mathbb Z$-grading on $\mathbb C[\mathbf z_w]$ with respect to which the generators $g^w_{k, \ell}$ for $J_{w, h}$ are homogeneous.
\end{lemma}

\begin{proof}
    It was shown in \cite[Section 2.3]{DSH} that the grading on $\mathbb C[\mathbf x_{w_0}]$ is positive.
    Since $\psi_w^{-1}(\mathbf z_w) = \mathbf x_{w_0} \setminus D_w$ is a subset of $\mathbf x_{w_0}$, it follows that the grading on $\mathbb C[\mathbf z_w]$ is also positive.
    Moreover, it was shown in \cite[Lemma 2.18]{DSH} that generators $f^{w_0}_{k, \ell}$ are homogeneous with respect to the non-standard grading on $\mathbb C[\mathbf x_{w_0}]$.
    We showed in Lemma \ref{lem:psiImage} that $g^w_{k,\ell} = \psi_w(f^{w_0}_{v_w(k), v_w(\ell)})$.
    Since $\psi_w$ is a homomorphism of rings, the image of the monomials in $f^{w_0}_{v_w(k), v_w(\ell)}$ under $\psi_w$ either retain their degree or get mapped to zero.
    In either case, we conclude that $g^w_{k, \ell}$ is homogeneous.
\end{proof}

In particular, from Lemma~\ref{lem:grading} we can conclude that the regular nilpotent Hessenberg Schubert cell ideals $J_{w, h}$ are homogeneous with respect to this non-standard grading.

We may now explicitly compute the Hilbert series of $R/J_{w,h}$ as follows. 

\begin{theorem}
Equip $R=\mathbb{C}[\mathbf{z}_w]$ with the positive $\mathbb{Z}$-grading defined in Section \ref{subsec:complete intersections}. 
Suppose that $h$ is an indecomposable Hessenberg function and $w \in {\rm Hess}(\mathsf N, h)^{\mathsf S}$. 
Then $\init_{<_n^w}(g^w_{k,\ell}) = -z_{n+1-v_w(k),v_w^{-1}(v_w(\ell)+1)}$ has degree $v_w(k) - v_w(\ell) - 1$ for all $k > h(\ell)$. 
Furthermore, the Hilbert series of $R/J_{w,h}$ is 
\[ 
H_{R/J_{w,h}}(t)
= \frac{
\displaystyle
\prod_{\substack{k > h(\ell) \\ v_w(k) > v_w(\ell) + 1}}
(1 - t^{v_w(k) - v_w(\ell) - 1})
}{
\displaystyle
\prod_{\substack{i < w(j) \\ j < w^{-1}(i)}} 
\left( 1 - t^{w(j) - i} \right)
}.
\]
\end{theorem}

\begin{proof}
Recall that the nonzero generators of $J_{w,h}$ are those $g^w_{k,\ell}$ satisfying $k > h(\ell)$ and $v_w(k) > v_w(\ell) + 1$.
By Lemmas \ref{lem:psiImage} and \ref{lem:compatibleInit}, $\init_{<_n^w}(g^w_{k,\ell}) = \psi_w(\init_{<_n}(f^{w_0}_{v_w(k),v_w(\ell)}))$. 
This initial term was computed in Lemma \ref{lem:f^w0-init}\eqref{lem-part:f^w0-init} as $\init_{<_n}(f^{w_0}_{v_w(k),v_w(\ell)}) = -x_{n+1-v_w(k), v_w(\ell) + 1}$.
Then by Definition \ref{def:psi} of $\psi_w$, \[\init_{<_n^w}(g^w_{k,\ell}) = \psi_w(-x_{n+1-v_w(k), v_w(\ell) + 1}) = -z_{n+1-v_w(k),v_w^{-1}(v_w(\ell)+1)}.\]
By definition of the $\mathbb Z$-grading given in Section \ref{subsec:complete intersections}, we have \[ \deg(-z_{n+1-v_w(k), v_w^{-1}(v_w(\ell)+1)}) = w_0(v_w(\ell) + 1) - (n + 1 - v_w(k)) = v_w(k) - v_w(\ell) - 1, \] 
since $w_0(v_w(\ell) + 1) = n + 1 - (v_w(\ell) + 1)$.
We argued in Lemma \ref{lem:grading} that the $g^w_{k, \ell}$ are homogeneous, so it follows that $\deg(g^w_{k, \ell}) = v_w(k) - v_w(\ell) - 1$ for all $k > h(\ell)$. This proves the first claim. 

Next we compute the Hilbert series of $R/J_{w,h}$. 
Since $\{g^w_{k,\ell} \mid k>h(\ell) \}$ is a Gr\"obner basis for the complete intersection ideal $J_{w,h}$ with respect to $<_n^w$, by Lemma \ref{lem: Hilbert Series}, we conclude 
\[ H_{R/J_{w,h}}(t) = H_{R}(t) \prod_{k > h(\ell)} \left(1 - t^{\deg(g_{k,\ell}^w)} \right). \]
But $H_{R}(t) =\displaystyle\prod_{z_{i,j}\in\mathbf{z}_w}(1-t^{\mu_{i,j}})^{-1}$, where $\mu_{i,j}=\deg(z_{i,j}) = w(j) - i$, so 
\[ 
H_{R/J_{w,h}}(t)
= \frac{
\displaystyle
\prod_{\substack{k > h(\ell) \\ v_w(k) > v_w(\ell) + 1}}
(1 - t^{v_w(k) - v_w(\ell) - 1})
}{
\displaystyle
\prod_{\substack{i < w(j) \\ j < w^{-1}(i)}} 
\left( 1 - t^{w(j) - i} \right)
}.
\]
where the product in the denominator is over those indices for which $z_{i,j} \in \mathbf z_w$, as given in \eqref{eqn:Omega_ij}.
\end{proof}

\begin{remark}
Since $R/J_{w,h}$ is isomorphic to a polynomial ring by the work in Section \ref{sec: affine paving}, we could have also computed the Hilbert series using the weights of the $z_{i,j}$ remaining after taking the quotient of $R$ by $J_{w,h}$. In other words, $H_{R/J_{w,h}}(t)$ is simply the series $\displaystyle\prod_{z_{i,j}\text{ $\nmid$ } \init(g^w_{k,\ell})}(1-t^{\mu_{i,j}})^{-1}$, where the product is over those indices $(i,j)$ for which $z_{i,j}\in\mathbf{z}_w$ does not appear as an initial term of any $g_{k,\ell}$ with $k>h(\ell)$.
\end{remark}

\subsection{Geometric vertex decomposability} \label{subsec: applications GVD}

Knutson, Miller, and Yong related the notion of a vertex decomposition of a simplicial complex to a classical algebraic geometry degeneration technique, and called this a {\bf geometric vertex decomposition} in the case that the degeneration is reduced \cite{KMY}.
Recently, Klein and the fourth author expanded this definition recursively to {\bf geometrically vertex decomposable} ideals, which provides a generalization of vertex decomposable simplicial complexes \cite{KR}.
The condition that a simplicial complex is vertex decomposable can be strengthened to require that the decomposition respects a fixed ordering of the vertices.
This strengthened notion is called compatibly vertex decomposable.
Klein and Rajchgot similarly introduced the notion of {\bf $<$-compatibly geometrically vertex decomposable}, where the degenerations respect a lexicographic monomial order $<$.
In particular, an ideal is $<$-compatibly geometrically vertex decomposable if and only if its initial ideal is the Stanley-Reisner ideal of a compatibly vertex decomposable simplicial complex with respect to a corresponding vertex order \cite[Proposition 2.14]{KR}.

While we refer the reader to \cite{KR} for formal definitions, we provide the following sufficient condition for being geometrically vertex decomposable, which is a straightforward generalization of \cite[Proposition 2.14]{KR}.

\begin{lemma}[{\cite[Lemma 3.6]{DSH}}] \label{lem:irr init implies gvd}
	Let $I \subseteq \mathbb C[\mathbf x]$ be an ideal in a polynomial ring and $<$ a lexicographic monomial order on $\mathbb C[\mathbf x]$.
	If $\init_<(I)$ is an ideal of indeterminates, then $I$ is $<$-compatibly geometrically vertex decomposable.
\end{lemma}

We saw in Theorem \ref{thm:subsetGB} that the initial ideal of the regular nilpotent Hessenberg Schubert cell ideal $J_{w,h}$ is an ideal of indeterminates, so the following is immediate from Lemma \ref{lem:irr init implies gvd} upon choosing the lexicographic monomial order $<_n^w$ from Definition \ref{def:cellMonomialOrder}.

\begin{proposition} \label{thm:GVDSchubertCell}
	Let $h$ be an indecomposable Hessenberg function and $w \in {\rm Hess}(\mathsf N, h)^{\mathsf S}$.
	Then $J_{w,h}$ is $<_n^w$-compatibly geometrically vertex decomposable.
\end{proposition}

\subsection{Frobenius splitting} \label{subsec: applications Frobenius}

Many well-studied subvarieties of the flag variety are known to be \emph{Frobenius split}, including Schubert varieties and Richardson varieties. 
Furthermore, each Schubert cell in the flag variety has a Frobenius splitting which compatibly splits all opposite Schubert varieties intersected with that Schubert cell. 

In this subsection, we observe that the situation is similar for regular nilpotent Hessenberg varieties. Namely, each Schubert cell in the flag variety has a Frobenius splitting which compatibly splits all the regular nilpotent Hessenberg varieties intersected with that Schubert cell.
This extends \cite[Corollary 5.14]{DSH}.

Let $\mathbb{K}$ be a algebraically closed field of characteristic $p>0$.
A \textbf{Frobenius splitting} of a polynomial ring $R = \mathbb{K}[y_1,\dots, y_n]$ is a map $\phi:R\rightarrow R$ which satisfies the following three conditions: (i) $\phi(a+b) = \phi(a)+\phi(b)$, for all $a,b\in R$, (ii) $\phi(a^pb) = a\phi(b)$, for all $a,b\in R$, and (iii) $\phi(1) = 1$. An ideal $I\subseteq R$ is \textbf{compatibly split} by the Frobenius splitting $\phi:R\rightarrow R$ if $\phi(I)\subseteq I$. We refer the reader to one of the many references on Frobenius splitting for further information on the subject (see e.g., \cite{BrionKumar, SmithZhang}). 

Let $\mathbb{K}[{\bf z}_w]$ denote the coordinate ring of the Schubert cell $X^w_\circ \subseteq GL_n(\mathbb{K})/B$. 
Consider each $g^w_{k,\ell}$ (which has integer coefficients) as an element of $\mathbb{K}[{\bf z}_w]$. Let ${J}_{w,h}\subseteq \mathbb{K}[{\bf z}_w]$ be the ideal generated by these polynomials $g^w_{k,\ell}\in \mathbb{K}[{\bf z}_w]$. Observe that  $J_{w,h}$ is the prime defining ideal of the regular nilpotent Hessenberg Schubert cell $\text{Hess}(\mathsf N,h)\cap X^w_\circ\subseteq GL_n(\mathbb{K})/B$. 

Since the coefficient of each  $\text{in}_{<^w_n}(g^w_{k,\ell})$ is not divisible by any prime $p>0$, we observe that the proof of Theorem \ref{thm: complete intersection} holds in positive characteristic. That is, $J_{w,h}\subseteq \mathbb{K}[{\bf z}_w]$ is a triangular complete intersection, minimally generated by the $g^w_{k,\ell}$. 
Let $G^w$ be the product of all of these minimal generators.
Let $Z$ denote the product of all indeterminates in the polynomial ring $\mathbb{K}[{\bf z}_w]$. 
Define $F^w := \frac{Z}{\text{in}_{<^w_n}(G^w)}G^w$, which is a polynomial since $\text{in}_{<^w_n}(G^w)$ is a squarefree monomial (up to sign), and hence divides $Z$. 
By construction, $G^w$ divides $F^w$. 
Additionally, $\text{in}_{<^w_n}(F^w)$ is equal to $Z$, the product of all the variables in $\mathbb{K}[{\bf z}_w]$.

Let $\mathrm{Tr}:\mathbb{K}[{\bf z}_w]\rightarrow \mathbb{K}[{\bf z}_w]$ be the additive map defined on monomials as follows: given a monomial $m\in \mathbb{K}[{\bf z}_w]$, define $\mathrm{Tr}(m)$ to be $\frac{(mZ)^{1/p}}{Z}$ if $mZ$ is a $p^{\rm th}$ power and $0$ otherwise. 
Given a polynomial $g\in \mathbb{K}[{\bf z}_w]$, define $\text{Tr}(g\hspace{1mm}\bullet)$ to be the map which sends $f\in \mathbb{K}[{\bf z}_w]$ to $\mathrm{Tr}(gf)$.

\begin{theorem}
The map $\mathrm{Tr}\left((F^w)^{p-1}\bullet \right)$ is a Frobenius splitting of $\mathbb{K}[{\bf z}_w]$ which compatibly splits each ideal $J_{w,h}\subseteq \mathbb{K}[{\bf z}_w]$.
\end{theorem}

\begin{proof}
This is immediate from \cite[Theorem 5.8 and Corollary 5.9]{DSH}.
\end{proof}

\bibliographystyle{plain}

\end{document}